\documentclass[12pt]{amsart}
\usepackage{amssymb}
\usepackage{amsmath}
\numberwithin{equation}{section}
\usepackage{breakcites}
\usepackage{color}
\usepackage{graphicx}
\usepackage{epstopdf}
\usepackage{float}
\usepackage{hyperref}

\newtheorem{thm}{Theorem}
\newtheorem{lemma}[thm]{Lemma}
\newtheorem{cor}[thm]{Corollary}
\newtheorem{prop}[thm]{Proposition}

\theoremstyle{definition}

\theoremstyle{remark}

\def\x{{\mathbf x}}
\def\K{{\mathbf K}}
\def\P{{\mathbb P}}

\begin{document}

\title[Schur positivity of trees]{A counterexample to a conjecture on Schur positivity of chromatic symmetric functions of trees}

\author{Emmanuella Sandratra Rambeloson}
\address{African Institute of Mathematical Sciences \\ Accra, Ghana}
\email{emmanuella@aims.ac.za}

\author{John Shareshian}
\address{Washington University \\ St Louis, MO, USA}
\email{jshareshian@wustl.edu}

\maketitle

\begin{abstract}
We show that no tree on twenty vertices with maximum degree ten has Schur positive chromatic symmetric function, thereby providing a counterexample to a conjecture from \cite{DSW}.
\end{abstract}

Among the many nice results on chromatic symmetric functions in the paper \cite{DSW} of Dahlberg, She and van Willigenburg is Theorem 39 therein, which says that no bipartite graph on $n$ vertices with a vertex of degree more than $\lceil \frac{n}{2} \rceil$ has Schur positive chromatic symmetric function.  In particular, Theorem 39 applies to trees.  A near-converse to Theorem 39 for trees is posed in \cite[Conjecture 42]{DSW}, which says that for every $n \geq 2$, there is a tree $T$ on $n$ vertices, one of which has degree $\lfloor \frac{n}{2} \rfloor$, such that the chromatic symmetric function of $T$ is Schur positive.  The authors of \cite{DSW} confirmed this conjecture for $n \leq 19$, using computer calculations.  Sadly, the conjecture is false for $n=20$, as we show here. We use SageMath \cite{SAGE} calculations after a preparatory proposition that reduces the number of trees that we must examine.

We give the requisite definitions and reiterate more formally.  Given a (finite, loopless, simple) graph $G=(V,E)$, a {\sf proper coloring} of $G$ is a function $\kappa$ from $V$ to the set $\P$ of positive integers such that $\kappa(v) \neq \kappa(w)$ whenever $\{v,w\} \in E$.  We fix an infinite set $\x:=\{x_i:i \in \P\}$ of pairwise commuting variables, and write $\K(G)$ for the set of all proper colorings of $G$. To each proper coloring $\kappa$ one associates a monomial 
$$
\x^\kappa:=\prod_{v \in V}x_{\kappa(v)}.
$$
The {\sf chromatic symmetric function} $X_G$ of $G$ is the sum of all such monomials,
$$
X_G(\x):=\sum_{\kappa \in \K(G)}\x^\kappa.
$$
Chromatic symmetric functions were introduced by Stanley in \cite{StCSF} and have drawn considerable attention.  Various results and conjectures, including the above-mentioned theorem and conjecture from \cite{DSW}, relate the structure of $G$ to the expansion of $X_G$ in terms of one or more familiar bases for the algebra $\Lambda$ of symmetric functions.   Recall that if $B$ is a basis for $\Lambda$ and $f \in \Lambda$, we call $f$ {\sf $B$-positive} if, when we expand $f=\sum_{b \in B}\alpha_b b$, each $\alpha_b$ is non-negative.  The {\sf Schur basis} for $\Lambda$ is a fundamental object in symmetric function theory. See for example \cite[Chapter 7]{StEC} for basic properties of Schur functions and other rudimentary facts about symmetric functions that will be used herein without reference.

%\begin{thm}[Theorem 39 of \cite{DSW}] \label{dswth}
%If $G=(V,E)$ is bipartite with $|V|=n$ and some $v \in V$ has degree more than $\lceil \frac{n}{2} \rceil$, then $X_G(\x)$ is not Schur positive.
%\end{thm}

%\begin{con}[Conjecture 42 of \cite{DSW}] \label{dswcon}
%For every $n \geq 2$, there is a tree $T$ with $n$ vertices, at least one of which has degree $\lfloor \frac{n}{2} \rfloor$, such that $X_T(\x)$ is Schur positive.
%\end{con}

We prove the following result, thereby disproving Conjecture 42 of \cite{DSW}.

\begin{thm} \label{main}
If $T$ is a tree on twenty vertices, one of which has degree ten, then $X_T(\x)$ is not Schur positive.
\end{thm}

A {\sf stable partition} of $G$ is a set partition $\pi:V=\bigcup_{j=1}^k \pi_j$ with each $\pi_j$ an independent set in $G$.  We assume without loss of generality that $|\pi_j| \geq |\pi_{j+1}|$ for each $j \in [n-1]$.   Setting $\lambda_j=|\pi_j|$ for each $j$, we get that $\lambda:=(\lambda_1,\ldots,\lambda_k)$ is a partition of the integer $|V|$. We call $\lambda$ the {\sf type} of $\pi$.  Given another partition $\mu=(\mu_1,\ldots,\mu_\ell)$ of $|V|$, we write $\mu \preceq \lambda$ if $\lambda$ {\sf dominates} $\mu$, that is, if $\sum_{j=1}^m \mu_j \leq \sum_{j=1}^m \lambda_j$ for all $m \in [k]$.  Our proof of Theorem \ref{main} rests on the following basic result, due to Stanley.  This result follows quickly from the fact that if $\mu \preceq \lambda$, then when the Schur function $s_\lambda$ is expanded in the monomial basis, the coefficient of $m_\mu$ is positive.
\begin{lemma}[Proposition 1.5 of \cite{StCSF2}] \label{domlem}
If $X_G(\x)$ is Schur positive and $G$ admits a stable partition of type $\lambda$, then $G$ admits a stable partition of type $\mu$ whenever $\mu \preceq \lambda$.
\end{lemma}

\begin{cor} \label{maincor}
Assume that $T=(V,E)$ is a tree on $2n$ vertices and $v \in V$ has degree $n$ in $T$.  If $X_T(\x)$ is Schur positive, then every $x \in V$ that is neither $v$ nor a neighbor of $v$ is a leaf in $T$.
\end{cor}

\begin{proof}
As $T$ is connected and bipartite, $T$ has a unique bipartition $\pi:V=\pi_1 \cup \pi_2$.  If $X_T(\x)$ is Schur positive, then $\pi$ has type $(n,n)$ by Lemma \ref{domlem}.  We assume without loss of generality that $v \in \pi_1$.  Then the neighborhood $N_T(v)$ is contained in $\pi_2$ and so $\pi_2=N_T(v)$.  Were the claim of the corollary false, some $z \in V$ would be at distance three from $v$ in $T$ and therefore lie in $\pi_2$, which is impossible.     
\end{proof}

For each partition $\nu=(\nu_1,\ldots,\nu_t)$ of $n-1$, let $T(\nu)$ be a tree on $2n$ vertices in which one vertex $v$ has exactly $n$ neighbors $v_1,\ldots,v_n$, and for $1 \leq i \leq t$, $v_i$ has exactly $\nu_i$ neighbors other than $v$ (each of which is necessarily a leaf).  The next result follows immediately from Corollary \ref{maincor}.

\begin{cor} \label{maincor2}
If $T$ is a tree on $2n$ vertices, one of which has degree $n$, and $X_T(\x)$ is Schur positive, then there is some partition $\nu$ of $n-1$ such that $T$ is isomorphic with $T(\nu)$.
\end{cor}

Theorem \ref{main} follows from the next result, which we prove by inspection using SageMath calculations.

\begin{prop} \label{lastprop}
If $\nu$ is a partition of the integer nine, then $X_{T(\nu)}$ is not Schur positive.
\end{prop}

Our computations reveal in particular that if $n=10$ and $\nu_1 \geq 6$, then the coefficient of $s_{(9,9,2)}$ in the Schur expansion of $X_{T(\nu)}(\x)$ is negative; and if $n=10$ and $\nu_1 \leq 5$, then the coefficient of $s_{(3,3,2,2,2,2,2,2,2)}$ in the Schur expansion of $X_{T(\nu)}(\x)$ is negative.  This Schur expansion has can have as few as four negative coefficients (when $\nu$ is one of $(6,2,1)$, $(6,1,1,1)$ or $(5,4)$) and as many as thirty (when $\nu$ is one of $(2,2,2,2,1)$, $(2,2,2,1,1,1)$ or $(1,1,1,1,1,1,1,1,1)$).
%The number ${\mathsf Neg}(\nu)$ of negative coefficients in the Schur expansion of $X_{T(\nu)}(\x)$ for every partition $\nu$ of nine is given in the table below.  
%we write, for example, $1^22^23$ for the partition $(3,2,2,1,1)$.  
Our programs, along with the complete Schur expansion of $X_{T(\nu)}(\x)$ for each partition $\nu$ of nine, can be found at \url{https://github.com/emmanuellasa/Schur_Decomposition_20}.

%$$
%\begin{array}{|c|c|c|c|c|c|c|c|c|c|c|c|c|c|c|c|} 
%\nu & 9 & 81 & 72 & 71^2 & 63 & 621 & 61^3 & 54 & 531 & 52^2 & 521^2 & 51^4 & 4^21 & 432 & 431^2 \\ {\mathsf  Neg}(\nu) & & & & & & & & & & & & & & &  
%\end{array}
%$$

\section*{Acknowledgement} We thank Stephanie van Willigenburg for helpful comments.

\end{document}